\documentclass{amsart}

\usepackage{amssymb}
\usepackage{graphicx}
\usepackage[cmtip,all]{xy}

\newtheorem{theorem}{Theorem}
\newtheorem{lemma}{Lemma}[section]
\newtheorem{cor}[theorem]{Corollary}

\theoremstyle{definition}

\theoremstyle{remark}
\newtheorem{remark}{Remark}[section]

\numberwithin{equation}{section}

\newcommand{\abs}[1]{\left\lvert#1\right\rvert}

\newcommand{\CC}{\mathcal{C}}

\newcommand{\CI}{\mathcal{I}}

\newcommand{\CL}{\mathcal{L}}
\newcommand{\CM}{\mathcal{M}}
\newcommand{\CN}{\mathcal{N}}
\newcommand{\CS}{\mathcal{S}}

\begin{document}

\title[Lehmer points and visible points on affine varieties]
{Lehmer points and visible points on affine varieties over finite fields}

\author{Kit-Ho Mak}
\address{School of Mathematics \\
Georgia Institute of Technology \\
686 Cherry Street \\
Atlanta, Georgia 30332, USA}
\email{kmak6@math.gatech.edu}

\author{Alexandru Zaharescu}
\address{Department of Mathematics \\
University of Illinois at Urbana-Champaign \\
273 Altgeld Hall, MC-382 \\
1409 W. Green Street \\
Urbana, Illinois 61801, USA}
\email{zaharesc@math.uiuc.edu}


\thanks{The second author is supported by NSF grant number DMS - 0901621.}

\begin{abstract}
Let $V$ be an absolutely irreducible affine variety over $\mathbb{F}_p$. A Lehmer point on $V$ is a point whose coordinates satisfy some prescribed congruence conditions, and a visible point is one whose coordinates are relatively prime. Asymptotic results for the number of Lehmer points and visible points on $V$ are obtained, and the distribution of visible points into different congruence classes is investigated.
\end{abstract}

\maketitle

\section{Introduction}

Let $p$ be a fixed large prime number. D.H. Lehmer raised the question of investigating the number $r(p)$ of integers $a\in\{1,2,\ldots,p-1\}$ for which $a$ and its multiplicative inverse $\overline{a}$ modulo $p$ are of opposite parity (see Guy \cite[Problem F12]{Guy94}). The problem was solved by Wenpeng Zhang in \cite{Zha93,Zha94b,Zha94a}. He proved that
\begin{equation*}
r(p)=\frac{p}{2}+O(\sqrt{p}\log^2{p}),
\end{equation*}
and then generalized this relation to the case when $p$ is replaced by any odd number $q>1$. He then defined a \textit{D. H. Lehmer number} to be any integer $a$ with $0<a<q$, coprime to $q$, and such that $a$ and $\overline{a}$ have opposite parity, and studied the distribution of the pairs of Lehmer numbers $(a,\overline{a})$ \cite{Zha94b}. The number $F_q(x,y)$ of such pairs inside the box $[1,xq]\times[1,yq]$ is given by
\begin{equation*}
F_q(x,y)=\frac{1}{2}xy\varphi(q)+O(\sqrt{q}d^2(q)\log^2{q}),
\end{equation*}
where $d(q)$ denotes the number of divisors of $q$.

Several generalizations have been considered. Instead of Lehmer pairs $(a,\overline{a})$ with opposite parity, \cite{ASZ06} considered \textit{Lehmer $k$-tuples}, which are defined as $(k+1)-$tuples $(n_1,\ldots,n_k,\overline{n_1\ldots n_k})$ modulo $q$ that satisfy the congruences $n_j\equiv b_j\pmod{a_j}$ and $\overline{n_1\ldots n_k}\equiv b_{k+1}\pmod{a_{k+1}}$ for some fixed $\mathbf{a}=(a_1,\ldots,a_{k+1})$, $\mathbf{b}=(b_1,\ldots,b_{k+1})\in\mathbb{Z}^{k+1}$ with $(a_1\ldots a_{k+1},q)=1$, $a_1,\ldots,a_{k+1}\geq 1$. Denote the number of Lehmer $k$-tuples by $N(\mathbf{a},\mathbf{b},q)$. It was shown in \cite{ASZ06} that
\begin{equation*}
N(\mathbf{a},\mathbf{b},q)=\frac{\varphi(q)^k}{a_1\ldots a_{k+1}}+O_{k,\epsilon}(q^{k-\frac{1}{2}+\epsilon}),
\end{equation*}
and a similar formula holds for the number of Lehmer points inside a region $\Omega\subseteq[0,q-1]^k$ with piecewise smooth boundary. These results were strengthened by Shparlinski in \cite{Shp09}. On the other hand, pairs $(x,Ax^k)$ of opposite parity, where $A$ and $k$ are arbitrary integers, were considered by Bourgain, Cochrane, Paulhus and Pinner in \cite{BCPP11}.

A different generalization is given in \cite{CoZa01}, where an absolutely irreducible algebraic curve $\CC$ of degree $d$ defined over the finite field $\mathbb{F}_p$ was considered. Let $\mathbf{a}=(a_1,\ldots,a_r)$, $\mathbf{b}=(b_1,\ldots,b_r)\in\mathbb{Z}^r$ with $a_1,\ldots,a_r\geq 1$. A \textit{Lehmer point} was defined as an $\mathbf{x}=(x_1\ldots,x_r)$ with $0\leq x_1<p$, such that $\mathbf{x}\in\CC$ and $x_j\equiv b_j\pmod{a_j}$, for all $1\leq j\leq r$. Denote by $\CL(p,r,\CC,\mathbf{a},\mathbf{b})$ the set of all Lehmer points. In \cite{CoZa01} it was shown that
\begin{equation*}
\#\CL(p,r,\CC,\mathbf{a},\mathbf{b})=\frac{p}{a_1\ldots a_r}+O_{r,d}(\sqrt{p}\log^r{p}).
\end{equation*}

In the present paper, we provide a common generalization of \cite{ASZ06,BCPP11,CoZa01,Zha93,Zha94a} by considering an absolutely irreducible affine variety $V\subseteq\mathbb{A}^r_p:=\mathbb{A}^r(\mathbb{F}_p)$ of dimension $n$ and degree $d$, embedded in an affine $r$-space ($r\geq 2$), which is not contained in any hyperplane. We are interested to see how these Lehmer points are distributed inside the space $[0,p-1]^r$. In particular, for any region $\Omega\subseteq[0,p-1]^r$ with piecewise smooth boundary, we will obtain asymptotic formulas for the number of Lehmer points on $V$ inside $\Omega$. Our results show that the rational points on $V$ are uniformly distributed among each congruence class.

Next, using the theory of Lehmer points, we go on to consider the number of \textit{visible points} on an affine variety $V$. Fix an embedding $V\subseteq\mathbb{A}^r_p$. By definition, a point $(x_1,\ldots,x_r)\in V$ with $0\leq x_j\leq p-1$ is \textit{visible} if the greatest common divisor of $x_1,\ldots,x_r$ is $1$. Geometrically, these are the points that an observer standing at the origin can see (that is, those points are not ``blocked'' by other integral points). Visible points on some special varieties over finite fields have been considered previously. Examples are those of plane curves investigated by Shparlinski and Voloch in \cite{ShVo07}, modular hyperbolas and its higher dimensional generalizations studied by Shparlinski and Winterhof \cite{Shp06, ShWi08}, and the modular exponential curves studied by Chan and Shparlinski in \cite{ChSh10}. Recently, the visibility question for points on curves of the form $y=f(x)$, where $f\in\mathbb{F}_p[x]$, is settled by Cilleruelo, Garaev, Ostafe and Shparlinski \cite{CGOS11}.

In our paper, we will treat the general problem of finding the number of visible points in an affine variety over finite fields, and obtain asymptotic formulas for the number of visible points whenever possible. For the cases that we cannot get an asymptotic formula for an individual $V$, an averaging result is obtained. This shows that almost all $V$ have the expected number of visible points. We remark that the study of visible points on the modular hyperbola is useful for certain approximation problems of real numbers by sums of rationals, see Chan \cite{Cha08, Cha09}. It would be interesting to see if the results of the present paper have similar applications.

Finally, we will patch the two concepts together and consider ``visible Lehmer points''. These are visible points on $V$ with prescribed congruence conditions at each coordinate. Unlike the case of Lehmer points, we cannot expect the visible points to lie uniformly in each congruence class due to the relatively prime condition on the points, but we may ask if the points lie uniformly on other congruence classes. We will prove that this is the case when the modulus at all coordinates is the same. If the moduli are different at different coordinates, the distribution may not be uniform.

\section{Statements of main results}

\subsection{Lehmer points}

Let $p$ be a large prime number. We will follow \cite{CoZa01} for the notion of Lehmer points. We let $\mathbf{a}=(a_1,\ldots,a_r)$, $\mathbf{b}=(b_1,\ldots,b_r)\in\mathbb{Z}^r$ with $a_1,\ldots,a_r\geq 1$ and none of the $a_j$ is a multiple of $p$. We say that an $\mathbf{x}=(x_1,\ldots,x_r)\in V$, $0\leq x_j\leq p-1$, is a \textit{Lehmer point} on $V$ with respect to $p$, $r$, $\mathbf{a}$, $\mathbf{b}$ if $x_j\equiv b_j\pmod{a_j}$ for all $1\leq j\leq r$. Let $\Omega$ be a region inside $[0,1)^r$ with piecewise smooth boundary, and let $\CL_{\Omega}=\CL_{\Omega}(p,r,V,\mathbf{a},\mathbf{b})$ be the set of Lehmer points inside the diluted region $p\Omega$. Asymptotic results for the number of Lehmer points when $\Omega$ is a box, and when $\Omega$ is a general region with piecewise smooth boundary, are as follows.

\begin{theorem}\label{mainthm}
Let $p$ be a prime, let $V$ be an absolutely irreducible affine variety in $\mathbb{A}^r_p$ of dimension $n$ and degree $d$, not contained in any hyperplane. Let $\mathbf{a}=(a_1,\ldots,a_r)$, with $a_1,\ldots,a_r\geq 1$ and $(a_j,p)=1$, $\mathbf{b}=(b_1,\ldots,b_r)\in\mathbb{Z}^r$. Then
\begin{enumerate}
\item if $\CI_1,\ldots,\CI_r\subseteq[0,1)$ are intervals and $\Omega=\CI_1\times\ldots\times\CI_r$, we have
\begin{equation*}
\#\CL_{\Omega}(p,r,V,\mathbf{a},\mathbf{b})={\rm vol}(\Omega)\cdot\frac{p^{n}}{a_1\ldots a_r} +O(2^r(4d+9)^{2n+1}p^{n-\frac{1}{2}}\log^r{p}).
\end{equation*}
\item For a general region $\Omega\subseteq[0,p-1]^r$ with piecewise smooth boundary, we have
\begin{equation*}
\#\CL_{\Omega}(p,r,V,\mathbf{a},\mathbf{b})={\rm vol}(\Omega)\cdot\frac{p^{n}}{a_1\ldots a_r}+O_{\Omega}((4d+9)^{\frac{2n+1}{r}}p^{n-\frac{1}{2r}}\log{p}).
\end{equation*}
\end{enumerate}
\end{theorem}

In arithmetic language, the above results say that inside any reasonable region $p\Omega$ with $\Omega\subseteq[0,1)^r$, the solutions of the following system of congruence equations
\begin{align*}
f_1(x_1,\ldots,x_r)&\equiv 0 \pmod{p} \\
\vdots & \\
f_m(x_1,\ldots,x_r)&\equiv 0 \pmod{p}
\end{align*}
where $f_i\in\mathbb{Z}[x_1,\ldots,x_r]$ are polynomials of degree at least $2$, are uniformly distributed among each congruence class, as long as the above system defines a non-planar absolutely irreducible affine variety in the affine $r$-space over $\mathbb{F}_p$.

\begin{remark}
The assumption that $V$ is not contained in any hyperplane is necessary to ensure that the congruences $x_j\equiv b_j\pmod{a_j}$ are independent. For example, let $V=\{(x,y)|\,x=y\}$ in $\mathbb{A}^2_p$. If we take $a_1=a_2=2$, $b_1=1$ but $b_2=0$, then we are looking for points $(x,y)$ with $x=y$, $x$ odd but $y$ even. There are no such points. On the contrary, if we take $a_1=a_2=2$, $b_1=b_2=1$, then we are looking for points $(x,y)$ with $x=y$, $x,y$ both odd. The number of such points is $p/2+O(1)$, which is much larger than the main term in Theorem \ref{mainthm}. This assumption is, however, not an important one. If the variety $V$ is contained in a hyperplane of $\mathbb{A}_p^r$, then by a linear change of variable we may assume the hyperplane is given by $x_r=0$. Hence we can embed $V$ in the affine $(r-1)$-space with variables $x_1,\ldots,x_{r-1}$.
\end{remark}

\begin{remark}
The classical Lehmer problem corresponds to the case $r=2$, $\mathbf{a}=(2,2)$, $\mathbf{b}=(0,1)$, and $V$ is the curve in $\mathbb{A}^2_p$ defined by $xy=1$. Our result reduces to that of \cite{Zha93} in this case. The Lehmer $k$-tuples considered in \cite{ASZ06} corresponds to the Lehmer points on $V$ where $V$ is the variety $x_1x_2\ldots x_r=1$.
\end{remark}

\begin{remark}
By putting $a_1=\ldots=a_r=1$, we recover the result of Fujiwara \cite{Fuj88} that the rational points on $V$ are uniformly distributed. Fujiwara's result was strengthened in sequel by Shparlinski and Skorobogatov \cite{ShSk90}, Skorobogatov \cite{Sko92}, Luo \cite{Luo99}, and Fouvry \cite{Fou00}.
\end{remark}

\subsection{Visible points}\label{subsecvis}

By a \textit{visible point} on $V$ we mean a point $(x_1,\ldots,x_r)$ on $V$ with $0\leq x_j\leq p-1$ and $\text{GCD}(x_1,\ldots,x_r)=1$. Let $\Omega\subseteq[0,1)^r$ be a region with piecewise smooth boundary, and let $\CN_{\Omega}(p,r,V)$ be the number of visible points in $p\Omega\cap V$ in the given embedding $V\subseteq\mathbb{A}^r_p$. If $\Omega=\CI_1\times\ldots\times\CI_r$ is a box, then we obtain the following result.
\begin{theorem}\label{mainthm2}
Let $p$ be a prime, let $V$ be an absolutely irreducible affine variety in $\mathbb{A}^r_p$ of dimension $n$ and degree $d$, not contained in any hyperplane. Let $\CI_1,\ldots,\CI_r\subseteq[0,1)$ be intervals and let $\Omega=\CI_1\times\ldots\times\CI_r$.
Then if $n> r/2$,
\begin{equation*}
\CN_{\Omega}(p,r,V)={\rm vol}(\Omega)\cdot\frac{p^n}{\zeta(r)} +O_{V,r}(p^{\frac{r}{r+1}(n+\frac{1}{2})}\log^{r}p),
\end{equation*}
where $\zeta(s)$ is the Riemann-Zeta function.
\end{theorem}

The restriction that $n> r/2$ is a significant one, for example it prevents us from considering curves. Our next task is to see how much can we relax this restriction. It turns out that under some very mild assumptions on $V$, we can completely remove this restriction when $n=\dim V \geq 2$. Before stating the theorem, we fix some notations (which follow \cite{Kat99}). Assume that $\dim V\geq 2$. We first homogenize $V$ using the variable $x_0$, call the resulting projective variety $X$. Define
\begin{equation}\label{defL}
L = \{ \mathbf{x}=(x_0,\ldots,x_r)\in X | x_0=0 \},
\end{equation}
and for any nonzero $\mathbf{u}=(u_1,\ldots,u_r)$, define
\begin{equation}\label{defH}
H_{\mathbf{u}} = \{ \mathbf{x}\in X | u_1x_1+\ldots+u_rx_r=0 \}.
\end{equation}
Suppose that $X\cap L\cap H_{\mathbf{u}}$ has dimension $n-2$. Denote by $\delta_{\mathbf{u}}$ the dimension of its singular locus, i.e.
\begin{equation*}
\delta_{\mathbf{u}} = \dim(\text{Sing}(X\cap L\cap H_{\mathbf{u}})).
\end{equation*}
Here we adopt the convention that the empty variety has dimension $-1$. If $X\cap L\cap H_{\mathbf{u}}$ has dimension $n-2$ for all $\mathbf{u}$ (this is so if $X\cap L$ is not contained in any hyperplane other than $L$), we define
\begin{equation}\label{defdelta}
\delta=\max_{\mathbf{u}\neq 0}\delta_{\mathbf{u}}.
\end{equation}
Now we can state our result.

\begin{theorem}\label{mainthm3}
Let $p$ be a prime, let $V$ be an absolutely irreducible affine variety in $\mathbb{A}^r_p$ of dimension $n\geq 2$ and degree $d$, not contained in any hyperplane. Let $\CI_1,\ldots,\CI_r\subseteq[0,1)$ be intervals and let $\Omega=\CI_1\times\ldots\times\CI_r$. Let $X$, $L$, $H_{\mathbf{u}}$ be as above.
Suppose that $X\cap L\cap H_{\mathbf{u}}$ has dimension $n-2$ for all $\mathbf{u}$ and $\delta$ is defined as in \eqref{defdelta}. Then
\begin{equation*}
\CN_{\Omega}(p,r,V)={\rm vol}(\Omega)\cdot\frac{p^n}{\zeta(r)} +O_V(p^{n-\frac{1}{2}}) +O_{V,r}(p^{\frac{(n+3+\delta)r}{2(r+1)}}\log^{r}p).
\end{equation*}
\end{theorem}

Note that the main term in the above theorem dominates the $O$-terms if $\delta \leq n-3$.

\begin{remark}
For a general region $\Omega$ with piecewise smooth boundary, we can proceed as in the proof of the second formula of Theorem \ref{mainthm} to obtain a formula for the number of visible points in $\Omega$. The analogue for Theorem \ref{mainthm2} is
\begin{equation*}
\CN_{\Omega}(p,r,V)={\rm vol}(\Omega)\cdot\frac{p^n}{\zeta(r)} +O_{V,r,\Omega}(p^{n-\frac{1}{2(r+1)}}\log{p}),
\end{equation*}
and the analogue for Theorem \ref{mainthm3} is
\begin{equation*}
\CN_{\Omega}(p,r,V)={\rm vol}(\Omega)\cdot\frac{p^n}{\zeta(r)} +O_V(p^{n-\frac{1}{2}}) +O_{V,r,\Omega}(p^{n-\frac{r(n-3-\delta)+2n}{2r(r+1)}}\log{p}).
\end{equation*}
We will omit the details.
\end{remark}

\begin{remark}
The assumption that $V$ is not contained in any hyperplane of $\mathbb{A}_p^r$ is still necessary in the case of visible points. For example, consider any absolutely irreducible affine variety $V\subseteq \mathbb{A}_p^r$. Instead of embedding $V$ into the affine $r$-space, we embed it into the affine $r+1$-space by appending a $1$ at the last coordinate. Then all points on $V$ become visible under the new embedding.
\end{remark}

The above theorems cover most varieties $V$ that one usually encounters. However, some important cases, such as the case when $V$ is a curve, are still not covered. In these cases we expect that Theorem \ref{mainthm2} should still hold true. Partial evidence is given by the following averaging result. As an immediate consequence, we see that Theorem \ref{mainthm2} is true for almost all $V$, regardless of its dimension and the dimension of the embedding space. We remark that such an averaging result for plane curves has been obtained by Shparlinski and Voloch in \cite{ShVo07}.

Suppose the variety $V$ is defined by a system of $m$ equations
\begin{align*}
f_1(x_1,\ldots,x_r)&\equiv 0 \pmod{p} \\
\vdots & \\
f_m(x_1,\ldots,x_r)&\equiv 0 \pmod{p}.
\end{align*}
Suppose $\mathbf{c}=(c_1,\ldots,c_m)\in\mathbb{F}^m_p$ is a vector. We let $V_{\mathbf{c}}$ to be the variety defined by
\begin{align*}
f_1(x_1,\ldots,x_r)&\equiv c_1 \pmod{p} \\
\vdots & \\
f_m(x_1,\ldots,x_r)&\equiv c_m \pmod{p}.
\end{align*}
It is not difficult to show that almost all $V_{\mathbf{c}}$ are absolutely irreducible. In fact, it can be shown that the set
\[\{ \mathbf{c}\in\mathbb{A}_p^m | V_{\mathbf{c}} \text{~is not absolutely irreducible} \} \]
is a Zariski closed set in $\mathbb{A}_p^m$ which does not contain the origin. Hence, the number of $\mathbf{c}$ such that $V_{\mathbf{c}}$ is not absolutely irreducible is $O(p^{m-1})$. Denote by $\CN(\mathbf{c})=\CN_{\Omega}(p,r,V_{\mathbf{c}})$ the number of visible points on $V_{\mathbf{c}}$ inside $p\Omega$.

\begin{theorem}\label{thmavg}
Let $p$ be a prime, let $V$ be an absolutely irreducible affine variety in $\mathbb{A}^r_p$ of dimension $n$ and degree $d$, defined by $m$ equations, and not contained in any hyperplane. Let $\CI_1,\ldots,\CI_r\subseteq[0,1)$ be intervals and let $\Omega=\CI_1\times\ldots\times\CI_r$. Then
\begin{multline*}
\sum_{\mathbf{c}\in\mathbb{F}_p^m} \abs{\CN(\mathbf{c})-{\rm vol}(\Omega)\cdot\frac{p^n}{\zeta(r)}} \ll_{V,r} p^{(n+m-\frac{1}{2})(1-\frac{1}{r})+1}\log^{r-1}{p} + p^r + p^{n+m-1}.
\end{multline*}
\end{theorem}

An immediate consequence of our averaging result is the following corollary, which says that the number of visible points are as expected for almost all $V_{\mathbf{c}}$.

\begin{cor}
Notations and assumptions are as in Theorem \ref{thmavg}. We have
\begin{equation*}
\CN(\mathbf{c})={\rm vol}(\Omega)\cdot\frac{p^n}{\zeta(r)}+o(p^n)
\end{equation*}
for all but $o(p^m)$ vectors $\mathbf{c}\in\mathbb{F}_p^m$.
\end{cor}

\subsection{Visible Lehmer points}

We now combine the concepts of Lehmer points and visible points, and consider visible Lehmer points on an affine variety $V$. The first observation is that the concepts of Lehmer points and visible points are not independent. For example, let $V\subseteq\mathbb{A}_p^2$ be a curve, and let $a_1=a_2=2$ and $b_1=b_2=0$. Then we are considering visible points on $V$ such that both coordinates are even. Clearly there are no such points. However, in this example, we may ask if the visible points are distributed uniformly among the other $3$ possible congruence classes.

More generally, assume $a_1=\ldots=a_r=a$ for some integer $a<p$. There are no visible points on the congruence classes corresponding to $\mathbf{b}=(b_1,\ldots,b_r)$ when $\text{GCD}(b_1,\ldots,b_r,a)\neq 1$, but we may ask about the distribution of visible points in the other congruence classes. It turns out that the distribution is uniform among these congruence classes. Let $\Omega\subseteq[0,1)^r$ be a region with piecewise smooth boundary. Denote by $\CN'_{\Omega}(p,r,V,a,\mathbf{b})$ the number of visible Lehmer points on $V$ inside the diluted region $p\Omega$.

Before we state our theorem, it is convenient to introduce the following arithmetic function $\varphi_r(a):\mathbb{Z}\rightarrow\mathbb{Z}$, defined by
\begin{equation*}
\varphi_r(a)=a^r\prod_{\substack{q|a \\ q~\textup{prime}}}\left(1-\frac{1}{q^r}\right).
\end{equation*}
For $r=1$ this is the Euler $\varphi$-function. In general, $\phi_r(a)$ is the number of $r$-tuples $(b_1,\ldots,b_r)$ with $0\leq b_j <a$ such that $\text{GCD}(b_1,\ldots,b_r,a)=1$.

Now we are ready to state our result.
\begin{theorem}\label{thmvlp}
Let $p$ be a prime, let $V$ be an absolutely irreducible affine variety in $\mathbb{A}^r_p$ of dimension $n$ and degree $d$, not contained in any hyperplane. Let $\CI_1,\ldots,\CI_r\subseteq[0,1)$ be intervals and let $\Omega=\CI_1\times\ldots\times\CI_r$. Let $a<p$ be an integer, and let $\mathbf{b}=(b_1,\ldots,b_r)\in\mathbb{Z}^r$ be an integral vector. If $\text{GCD}(b_1,\ldots,b_r,a)\neq 1$, then
\begin{equation*}
\CN'_{\Omega}(p,r,V,a,\mathbf{b})=0.
\end{equation*}
If $\text{GCD}(b_1,\ldots,b_r,a)=1$ we have the following:
\begin{enumerate}
\item For $n>r/2$, we have
\begin{equation*}
\CN'_{\Omega}(p,r,V,a,\mathbf{b}) = {\rm vol}(\Omega)\cdot\frac{p^n}{\zeta(r)\varphi_r(a)} +O_{V,r}(p^{\frac{r}{r+1}(n+\frac{1}{2})}\log^{r-1}p).
\end{equation*}
\item For $V$ that satisfies the assumptions in Theorem \ref{mainthm3}, we have
\begin{equation*}
\CN'_{\Omega}(p,r,V,a,\mathbf{b}) = {\rm vol}(\Omega)\cdot\frac{p^n}{\zeta(r)\varphi_r(a)}
+O_V(p^{n-\frac{1}{2}}) +O_{V,r}(p^{\frac{(n+3+\delta)r}{2(r+1)}}\log^{r}p).
\end{equation*}
\end{enumerate}
\end{theorem}

\begin{remark}
If $\text{GCD}(b_1,\ldots,b_r,a)=1$, for a general region $\Omega$ with piecewise smooth boundary, we have
\begin{equation*}
\CN'_{\Omega}(p,r,V,a,\mathbf{b})={\rm vol}(\Omega)\cdot\frac{p^n}{\zeta(r)\varphi_r(a)} +O_{V,r,\Omega}(p^{n-\frac{1}{2(r+1)}}\log{p})
\end{equation*}
if $n>r/2$. On the other hand, if $V$ satisfies the assumptions in Theorem \ref{mainthm3}, then
\begin{equation*}
\CN'_{\Omega}(p,r,V,a,\mathbf{b})={\rm vol}(\Omega)\cdot\frac{p^n}{\zeta(r)\varphi_r(a)} +O_V(p^{n-\frac{1}{2}})
+O_{V,r,\Omega}(p^{n-\frac{r(n-3-\delta)+2n}{2r(r+1)}}\log{p}).
\end{equation*}
\end{remark}

\begin{remark}
If the modulus for each coordinate are different, say $a_1,\ldots,a_r$, then we can take the least common multiple of the $a_i$'s as the common modulus $a$ and write the original congruence condition in each coordinate as congruence conditions modulo $a$. By applying Theorem \ref{thmvlp} to these congruence conditions one by one, we can find the number of visible Lehmer points in this case as well.

However, one should not expect the distribution of visible points to be uniformly distributed into each congruence class (that can possibly have a visible point) if the modulus is different for each coordinate. For example, let $V=\mathbb{A}_p^2$, $r=2$ and $(a_1,a_2)=(2,3)$. We can take $a=6$. If $(b_1,b_2)=(1,0)$ the visible Lehmer points can fall into $4$ classes, namely
\[ (x,y)\equiv (1,0), (1,3), (5,0), (5,3) \pmod{6}.\]
On the other hand, if $(b_1,b_2)=(0,1)$, the visible Lehmer points can fall into only $3$ classes:
\[ (x,y)\equiv (0,1), (2,1), (4,1) \pmod{6}.\]
Therefore, by Theorem \ref{thmvlp}, the number of visible Lehmer points with respect to $(b_1,b_2)=(1,0)$ should be about $4/3$ times the number of visible Lehmer points with respect to $(b_1,b_2)=(0,1)$.
\end{remark}

\section{Preliminary results}

Let $e_p(x)=e^{2\pi ix/p}$. The following lemmas will be useful.

\begin{lemma}\label{lemexpsum}
Let $p$ be a prime, $\CI$ be an interval in $[0,1)$, and $0\leq b<a$ be integers such that $(a,p)=1$. Then
\begin{equation*}
\sum_{u\neq 0 \textup{~mod~$p$}}\abs{\sum_{ m\in p\CI, m\equiv b\textup{~(mod $a$)}}e_p(-um)}\leq 2p\log{p}.
\end{equation*}
\end{lemma}
\begin{proof}
Write $m=ak+b$. Let $\alpha$ and $\beta$ be such that $m\in p\CI$ if and only if $\alpha\leq k\leq\beta$. Then

\begin{align*}
\abs{\sum_{\substack{m\in p\CI \\ m\equiv b\textup{~(mod $a$)}}}e_p(-um)} =\abs{\sum_{k=\alpha}^{\beta}e_p(-uak-ub)}&=\abs{e_p(-ub)\sum_{k=\alpha}^{\beta}e_p(-uak)} \\
&=\abs{\sum_{k=\alpha}^{\beta}e_p(-uak)} \\
&=\abs{e_p(-ua\alpha)\frac{1-e_p(-ua(\beta-\alpha+1))}{1-e_p(-ua)}} \\
&\leq \frac{2}{\abs{1-e_p(-ua)}} \\
&=\frac{2}{\abs{2\sin\left(\frac{\pi ua}{p}\right)}} \\
&=\frac{1}{\abs{\sin\left(\frac{\pi s}{p}\right)}},
\end{align*}

where $s$ is the least absolute residue of $au$ modulo $p$. As $\abs{\frac{\pi s}{p}}\leq\frac{\pi}{2}$, we have
\begin{equation*}
\abs{\sin\left(\frac{\pi s}{p}\right)}\geq\frac{2}{\pi}\frac{\pi \abs{s}}{p}=\frac{2\abs{s}}{p}.
\end{equation*}
So,
\begin{equation}\label{eqnlem1}
\abs{\sum_{\substack{m\in p\CI \\ m\equiv b\textup{~(mod $a$)}}}e_p(-um)}\leq \frac{p}{\abs{s}}.
\end{equation}
Since $s\equiv au\pmod{p}$ and $(a,p)=1$, the nonzero $s$ and $u$ are in $1-1$ correspondence. Summing \eqref{eqnlem1} over all nonzero $u$ mod $p$ and using the inequality
\begin{equation*}
1+\frac{1}{2}+\ldots+\frac{1}{\frac{p-1}{2}}\leq\log{p},
\end{equation*}
we get the desired inequality.
\end{proof}

The second lemma is the Lang-Weil bound on the number of rational points on an irreducible affine variety. For the original proof see Lang-Weil \cite{LaWe54}.
\begin{lemma}\label{lemLW}
Let $V\subseteq\mathbb{A}_p^r$ be an irreducible affine variety over $\mathbb{F}_p$ of dimension $n$ and degree $d$, then
\begin{equation*}
\#V(\mathbb{F}_p)=p^n+O((d-1)(d-2)p^{n-\frac{1}{2}}).
\end{equation*}
\end{lemma}

The next lemma is the Bombieri-Deligne's estimate of exponential sums over an irreducible affine variety, which follows from the work of Deligne \cite{Del74,Del80} and Bombieri \cite{Bom78} (see also \cite{Bom66} for the case of a curve).
\begin{lemma}\label{lembom}
Let $V\subseteq\mathbb{A}_p^r$ be an irreducible affine variety over $\mathbb{F}_p$ of dimension $n$ and degree $d$, not contained in any hyperplane. If $(u_1,\ldots,u_r)$ is nonzero modulo $p$, then
\begin{equation*}
\abs{\sum_{(x_1,\ldots,x_r)\in V}e_p(u_1x_1+\ldots+u_rx_r)}\leq (4d+9)^{n+r}p^{n-\frac{1}{2}}.
\end{equation*}
\end{lemma}

The above estimate requires almost no assumptions on the variety $V$ (apart from that $V$ is absolutely irreducible). We can have a much better estimate if we assume some mild conditions on $V$, and use Katz's estimation \cite{Kat99}. Assume that $\dim V\geq 2$, we recall from Section \ref{subsecvis} that $X$ is the homogenization of $V$ using the variable $x_0$, and $L$, $H_{\mathbf{u}}$ is defined by \eqref{defL} and \eqref{defH} respectively. Let $\delta_{\mathbf{u}}$ denote the dimension of the singular locus of $X\cap L\cap H_{\mathbf{u}}$. We have the following estimate of exponential sums in terms of $\delta_{\mathbf{u}}$ \cite[Theorem 5]{Kat99}.

\begin{lemma}\label{lemKatz}
Let $V\subseteq\mathbb{A}_p^r$ be an irreducible affine variety over $\mathbb{F}_p$ of dimension $n$ and degree $d$, not contained in any hyperplane, and let $X$ be its homogenization. Let $\mathbf{u}=(u_1,\ldots,u_r)$ be a nonzero vector modulo $p$, and $L$, $H_{\mathbf{u}}$ as in \eqref{defL} and \eqref{defH}. Suppose that $X\cap L\cap H_{\mathbf{u}}$ has dimension $n-2$, and let $\delta_{\mathbf{u}}$ be the dimension of its singular locus, then
\begin{equation*}
\abs{\sum_{(x_1,\ldots,x_r)\in V}e_p(u_1x_1+\ldots+u_rx_r)}\leq (4d+9)^{n+r}p^{\frac{n+1+\delta_{\mathbf{u}}}{2}}.
\end{equation*}
\end{lemma}

\section{Lehmer points on a box: proof of Theorem \ref{mainthm}} \label{secmainthm}

From the orthogonality of exponential sums, we have
\begin{equation*}
\sum_{u_j \textup{~mod $p$}}e_p(u_j(n_j-m_j)) =
\begin{cases}
p, & \text{~if~} n_j=m_j, \\
0, & \text{~otherwise,}
\end{cases}
\end{equation*}
so if $\CI_j$ are intervals inside $[0,1)$, we have
\begin{equation*}
\sum_{\substack{m_j\in p\CI_j \\ m_j\equiv b_j\textup{~(mod $a_j$)}}}\sum_{u_j \textup{~mod $p$}}e_p(u_j(n_j-m_j)) =
\begin{cases}
p, & \text{~if~} n_j\equiv b_j \pmod{a_j} \text{~and~} n_j\in p\CI_j, \\
0, & \text{~otherwise.}
\end{cases}
\end{equation*}
Therefore, the number of Lehmer points on $V$ lying in a box $p\Omega$ with $\Omega=\CI_1\times\ldots\times\CI_r\subseteq[0,1)^r$ can be written as an exponential sum.
\begin{align*}
\#\CL_{\Omega}(p,r,V,\mathbf{a},\mathbf{b}) &= \frac{1}{p^r}\sum_{\substack{(x_1,\ldots,x_r)\in V \\ u_j \textup{~mod~$p$}\\ m_j\in p\CI_j, m_j\equiv b_j\textup{~(mod $a_j$)}}} e_p(u_1(x_1-m_1)+\ldots+u_r(x_r-m_r)) \\
&=\frac{1}{p^r}\sum_{u_j \textup{~mod~$p$}}\sum_{ m_j\in p\CI_j, m_j\equiv b_j\textup{~(mod $a_j$)}}e_p(-u_1m_1-\ldots-u_rm_r) \\
&\qquad\times\sum_{(x_1\ldots,x_r)\in V}e_p(u_1x_1+\ldots+u_rx_r) \\
&=\frac{1}{p^r}\sum_{u_1 \textup{~mod~$p$}}\cdots\sum_{u_r \textup{~mod~$p$}} \prod_{j=1}^r\left(\sum_{ m_j\in p\CI_j, m_j\equiv b_j\textup{~(mod $a_j$)}}e_p(-u_jm_j)\right) \\
&\qquad\times\sum_{(x_1\ldots,x_r)\in V}e_p(u_1x_1+\ldots+u_rx_r) \\
&=M+E,
\end{align*}
where $M$ is the sum of the terms with $(u_1,\ldots,u_r)=(0,\ldots,0)$ and $E$ is the sum of remaining terms.

For the main term $M$, we have
\begin{align}
M &= \frac{1}{p^r}\prod_{j=1}^r\left(\frac{\abs{p\CI_j}}{a_j}+O(1)\right)\#V(\mathbb{F}_p) \nonumber \\
&= \abs{\CI_1}\ldots\abs{\CI_r}\cdot\frac{\#V(\mathbb{F}_p)}{a_1\ldots a_r}(1+O(1/p)) \label{eqnmain} \\
&= {\rm vol}(\Omega)\cdot\frac{\#V(\mathbb{F}_p)}{a_1\ldots a_r}(1+O(1/p)).
\end{align}
Applying the Lang-Weil bound (Lemma \ref{lemLW}), we see that the main term is given by
\begin{equation}\label{eqnmain1}
M={\rm vol}(\Omega)\cdot\frac{p^n}{a_1\ldots a_r}+O((d-1)(d-2)p^{n-\frac{1}{2}}).
\end{equation}

The remaining terms can be estimated using Lemma \ref{lemexpsum} and Lemma \ref{lembom}.
\begin{align}
\abs{E} &\leq \frac{1}{p^r}(2p\log p+p)^r((4d+9)^{n+r}p^{n-\frac{1}{2}}) \nonumber \\
&\leq 2^r(4d+9)^{n+r} p^{n-\frac{1}{2}}(\log{p}+1)^r, \label{eqnerr1}
\end{align}
where in the first inequality, the factor $(2p\log p + p)^r$ is used instead of $(2p\log p)^r$ in order to bound also those terms where some (but not all) of the $u_j$'s are zero. Combining \eqref{eqnmain1} and \eqref{eqnerr1}, we obtain
\begin{equation*}
\#\CL_{\Omega}(p,r,V,\mathbf{a},\mathbf{b}) = {\rm vol}(\Omega)\cdot\frac{p^n}{a_1\ldots a_r}+O(2^r(4d+9)^{n+r}p^{n-\frac{1}{2}}\log^r{p}).
\end{equation*}

This proves the first formula of Theorem \ref{mainthm}. The second formula follows from the first one by a general theorem relating the box discrepancies with the discrepancies of general regions with piecewise smooth boundaries. We refer the reader to the papers of Laczkovich \cite{Lac95} and Weyl \cite{Wey39} for details.

\begin{remark}
If $V$ is an affine curve, we can use the bound in \cite[Theorem 6]{Bom66} instead of \cite{Bom78} to improve Lemma \ref{lembom}, thereby improving the error terms in Theorem \ref{mainthm}. In this case, we have
\begin{equation*}
\#\CL_{\Omega}(p,r,V,\mathbf{a},\mathbf{b}) = {\rm vol}(\Omega)\cdot\frac{p}{a_1\ldots a_r}+O(2^r d^2 p^{\frac{1}{2}}\log^r{p})
\end{equation*}
when $\Omega$ is a box $\CI_1\times\ldots\times\CI_r$, and for general region $\Omega$ with piecewise smooth boundary, we have
\begin{equation*}
\#\CL_{\Omega}={\rm vol}(\Omega)\cdot\frac{p}{a_1\ldots a_r}+O_{\Omega}(d^\frac{2}{r} p^{1-\frac{1}{2r}}\log{p}).
\end{equation*}
These formula also make precise the dependence on $r$ and $d$ for the corresponding formula in \cite[Theorem 1]{CoZa01}.
\end{remark}

\section{Visible points on $V$: proof of Theorem \ref{mainthm2} and Theorem \ref{mainthm3}} \label{secvis}

To simplify notation, we write $\CN_{\Omega}$ to be the number of visible points in $p\Omega$, with $p$, $r$ and $V$ understood. For any positive integer $d$, define
\begin{multline*}
\CM_{\Omega}(k)=\#\{ (x_1,\ldots,x_r)\in (V\cap\Omega)-(0,\ldots,0) | \\
0\leq x_j\leq p-1, k \text{~divides~} \text{GCD}(x_1,\ldots,x_r) \}.
\end{multline*}
Note that $\CM_{\Omega}(k)=0$ if $k>p$. We have
\begin{equation}\label{eqnCNCM}
\CN_{\Omega}=\sum_{k=1}^{\infty}\mu(k)\CM_{\Omega}(k),
\end{equation}
where $\mu(k)$ is the M\"{o}bius function. Note that $k$ divides GCD$(x_1,\ldots,x_r)$ means that $x_j\equiv 0\pmod{k}$, so if $k$ is small compared to $p$, we can use Theorem \ref{mainthm} to estimate $\CM_{\Omega}(k)$. We have
\begin{equation}\label{eqnCM}
\CM_{\Omega}(k)={\rm vol}(\Omega)\cdot\frac{p^n}{k^r}+O_{V,r}(p^{n-\frac{1}{2}}\log^r{p})
\end{equation}
for such $k$. We now fix a number $K<p$ whose value will be determined later. For $k\leq K$ we apply the estimate \eqref{eqnCM} to \eqref{eqnCNCM}, and obtain
\begin{equation}\label{eqnCN1}
\CN_{\Omega}={\rm vol}(\Omega)p^n\sum_{k\leq K}\frac{\mu(k)}{k^r}+O_{V,r}(Kp^{n-\frac{1}{2}}\log^r{p})+O\left(\sum_{K<k\leq p}\CM_{\Omega}(k)\right).
\end{equation}
The main term is
\begin{align*}
{\rm vol}(\Omega)p^n\sum_{k\leq K}\frac{\mu(k)}{k^r} &= \abs{\CI_1}\ldots\abs{\CI_r}p^n\left(\frac{1}{\zeta(r)}+O\left(\frac{1}{K^{r-1}}\right)\right) \\
&= {\rm vol}(\Omega)\cdot\frac{p^n}{\zeta(r)}+O\left(\frac{p^n}{K^{r-1}}\right).
\end{align*}

The last term of \eqref{eqnCN1} can be estimated as follows. For any $k$, a point $P=(x_1,\ldots,x_r)\in V\cap\Omega$ satisfies $k|\text{GCD}(x_1,\ldots,x_r)$ if and only if $P$ can be written as $P=(ky_1,\ldots,ky_r)$, where $(y_1,\ldots,y_r)\in\frac{1}{k}\Omega$. Note that for any $K<k\leq p$, we have $\frac{1}{k}\Omega\subseteq[1,p/K]^r$. 
Now fix a point $(y_1,\ldots,y_r)\in[1,p/K]^r$. If $V$ is defined by the polynomials $f_1(x_1,\ldots,x_r),\ldots,f_m(x_1,\ldots,x_r)$, then the values of $k$ so that $(ky_1,\ldots,k_r)\in V$ are those with
\begin{equation*}
f_j(ky_1,\ldots,ky_r)\equiv 0 \pmod{p}, ~1\leq j\leq m.
\end{equation*}
For a fix $(y_1,\ldots,y_r)$, the number of such $k$ is finite (it is at most the degree of $f_1$). Thus,
\begin{equation*}
\sum_{K<k\leq p}\CM_{\Omega}(k)=O\left(\frac{p^r}{K^r}\right).
\end{equation*}

We now put these back to \eqref{eqnCN1} gives
\begin{equation}\label{eqnCN2}
\CN_{\Omega}={\rm vol}(\Omega)\cdot\frac{p^n}{\zeta(r)} +O_{V,r}(Kp^{n-\frac{1}{2}}\log^r{p}) +O\left(\frac{p^r}{K^{r}}\right).
\end{equation}
Finally, we balance the error terms by choosing $K$ such that
\begin{equation*}
Kp^{n-\frac{1}{2}} = \frac{p^r}{K^{r}}.
\end{equation*}
This gives $K=p^{\frac{r-n+1/2}{r+1}}$. Inserting this back in \eqref{eqnCN2} yields
\begin{equation*}
\CN_{\Omega}(p,r,V)={\rm vol}(\Omega)\cdot\frac{p^n}{\zeta(r)} +O_{V,r}(p^{\frac{r}{r+1}(n+\frac{1}{2})}\log^{r}p).
\end{equation*}
To complete the proof of Theorem \ref{mainthm2} it remains to note that the main term dominates the $O$-terms when $n> r/2$.

The above estimate is not strong enough for $n\leq r/2$. To obtain a stronger estimation, notice that the error term in \eqref{eqnCM} comes from two parts, namely the Lang-Weil bound (Lemma \ref{lemLW}) and the Bombieri estimate (Lemma \ref{lembom}). To cope with the Lang-Weil bound we use \eqref{eqnmain} for the main term, and to improve the Bombieri estimate we will use Katz's estimate (Lemma \ref{lemKatz}). This explains why we require some mild conditions on $V$ in Theorem \ref{mainthm3}.

Using the idea in Section \ref{secmainthm}, we can write
\begin{equation*}
\CM_{\Omega}(k) = M+E,
\end{equation*}
where
\begin{equation*}
M={\rm vol}(\Omega)\cdot\frac{\#V(\mathbb{F}_p)}{k^r}(1+O(1/p)),
\end{equation*}
and
\begin{equation*}
E = \frac{1}{p^r}\prod_{j=1}^r\sum_{\substack{\mathbf{u}\neq 0\\ m_j\in p\CI_j, m_j\equiv b_j\textup{~(mod $a_j$)}}}e_p(-u_jm_j) \sum_{(x_1\ldots,x_r)\in V}e_p(u_1x_1+\ldots+u_rx_r),
\end{equation*}
where $\mathbf{u}=(u_1,\ldots,u_r)$. Let $X$ be the homogenization of $V$ by the variable $x_0$. Recall that
\begin{equation*}
L = \{ \mathbf{x}=(x_0,\ldots,x_r)\in X | x_0=0 \},
\end{equation*}
and for any nonzero $\mathbf{u}=(u_1,\ldots,u_r)$,
\begin{equation*}
H_{\mathbf{u}} = \{ \mathbf{x}\in X | u_1x_1+\ldots+u_rx_r=0 \}.
\end{equation*}
Suppose that $X\cap L\cap H_{\mathbf{u}}$ has dimension $n-2$ for all $\mathbf{u}$, then we can apply Lemma \ref{lemKatz}. Set
\[\delta=\max_{\mathbf{u}\neq 0}\delta_{\mathbf{u}},\]
then
\begin{align*}
E &= O(\frac{1}{p^r}(p\log p)^r (4d+9)^{n+r}p^{\frac{n+1+\delta}{2}}) \\
&= O_{V,r}(p^{\frac{n+1+\delta}{2}}\log^r p).
\end{align*}
Hence,
\begin{equation*}
\CM_{\Omega}(k) = {\rm vol}(\Omega)\cdot\frac{\#V(\mathbb{F}_p)}{k^r}+O_{V,r}(p^{\frac{n+1+\delta}{2}}\log^r p).
\end{equation*}
Using this and following the same calculation as in the first part of this section, we get
\begin{align*}
\CN_{\Omega}(p,r,V)&={\rm vol}(\Omega)\cdot\frac{\#V(\mathbb{F}_p)}{\zeta(r)} +O_{V,r}(p^{\frac{(n+3+\delta)r}{2(r+1)}}\log^{r}p) \\
&={\rm vol}(\Omega)\cdot\frac{p^n}{\zeta(r)} \\
&\qquad +O_{V}(p^{n-\frac{1}{2}}) +O_{V,r}(p^{\frac{(n+3+\delta)r}{2(r+1)}}\log^{r}p).
\end{align*}
Notice that if $\delta \leq n-3$, then the main term dominates the $O$-terms. This completes the proof of Theorem \ref{mainthm3}.

\section{An averaging result: proof of Theorem \ref{thmavg}}

Let $\CS$ be the set of $\mathbf{c}$ so that $V_{\mathbf{c}}$ is absolutely irreducible. Then from \eqref{eqnCN1} and the calculation of its main term, for any $\mathbf{c}\in\CS$, we have
\begin{multline}\label{eqnavg1}
\abs{\CN(\mathbf{c})-{\rm vol}(\Omega)\cdot\frac{p^n}{\zeta(r)}} \ll_{V,r} Kp^{n-\frac{1}{2}}\log^r{p} \\
+\frac{p^n}{K^{r-1}}+\sum_{K<k\leq p}\CM_{\Omega,V_{\mathbf{c}}}(k).
\end{multline}
Notice that for distinct $\mathbf{c}$, the varieties $V_{\mathbf{c}}$ are disjoint as sets. Therefore,
\begin{equation*}
\sum_{\mathbf{c}\in\mathbb{F}_p^m} \CM_{\Omega,V_{\mathbf{c}}}(k) \leq \frac{\abs{p\CI_1}\ldots\abs{p\CI_r}}{k^r},
\end{equation*}
and hence
\begin{equation*}
\sum_{\mathbf{c}\in\mathbb{F}_p^m}\sum_{K<k\leq p} \CM_{\Omega,V_{\mathbf{c}}}(k) \leq \sum_{K<k\leq p}\frac{\abs{p\CI_1}\ldots\abs{p\CI_r}}{k^r} \leq \frac{p^r}{K^{r-1}}.
\end{equation*}
Putting this into \eqref{eqnavg1}, we have
\begin{equation*}
\sum_{\mathbf{c}\in\CS} \abs{\CN(\mathbf{c})-{\rm vol}(\Omega)\cdot\frac{p^n}{\zeta(r)}} \ll_{V,r} Kp^{n+m-\frac{1}{2}}\log^r{p}+\frac{p^r}{K^{r-1}}.
\end{equation*}
We balance the error term by choosing $K=p^{1-\frac{1}{r}(n+m-\frac{1}{2})}\log^{-1}p$, and obtain
\begin{equation}\label{eqnCS1}
\sum_{\mathbf{c}\in\CS} \abs{\CN(\mathbf{c})-{\rm vol}(\Omega)\cdot\frac{p^n}{\zeta(r)}} \ll_{V,r} p^{(n+m-\frac{1}{2})(1-\frac{1}{r})+1}\log^{r-1}{p}.
\end{equation}
For $\mathbf{c}\notin\CS$, we estimate $\CN(\mathbf{c})$ trivially as
\begin{equation*}
\sum_{\mathbf{c}\notin\CS}\CN(\mathbf{c}) \leq p^r.
\end{equation*}
As $\abs{\CS}=O(p^{m-1})$, we have
\begin{equation}\label{eqnCS2}
\sum_{\mathbf{c}\notin\CS}\abs{\CN(\mathbf{c})-{\rm vol}(\Omega)\cdot\frac{p^n}{\zeta(r)}} \ll_{V} p^r + p^{n+m-1}.
\end{equation}
Combining \eqref{eqnCS1} and \eqref{eqnCS2}, we complete the proof of Theorem \ref{thmavg}.

\section{Visible Lehmer points: proof of Theorem \ref{thmvlp}}

Similar to Section \ref{secvis}, we write $\CN'_{\Omega}$ to be the number of visible Lehmer points in $\Omega$, with $p$, $r$, $V$ and $a$ understood. Define
\begin{multline}\label{eqndefCM'}
\CM'_{\Omega}(k)=\#\{ (x_1,\ldots,x_r)\in (V\cap\Omega)-(0,\ldots,0) |\\
0\leq x_j\leq p-1, k \text{~divides~} \text{GCD}(x_1,\ldots,x_r), x_j\equiv b_j\pmod{a} \}.
\end{multline}
We have
\begin{equation}\label{eqnCN'CM'}
\CN'_{\Omega}=\sum_{k=1}^{p-1}\mu(k)\CM'_{\Omega}(k).
\end{equation}
The conditions in \eqref{eqndefCM'} amount to
\begin{equation}\label{sysCM'}
\begin{aligned}
x_j &\equiv 0 \pmod{k}, \\
x_j &\equiv b_j \pmod{a}.
\end{aligned}
\end{equation}
Let $g=\text{GCD}(k,a)$. If $g=1$ \eqref{sysCM'} is equivalent to a system of congruences modulo $ka$, while if $g>1$ \eqref{sysCM'} has no solution if $b_j\not\equiv 0\pmod{g}$, and is equivalent to a system of congruences modulo $ka/g$ if $b_j\equiv 0 \pmod{g}$. Since we assume that $\text{GCD}(b_1,\ldots,b_r,a)=1$, we have $\CM'_{\Omega}(k)=0$ if $g>1$. For $g=1$, we use Theorem \ref{mainthm} to obtain
\begin{equation*}
\CM'_{\Omega}(k) = {\rm vol}(\Omega)\cdot\frac{p^n}{(ka)^r}+O_{V,r}(p^{n-\frac{1}{2}}\log^r{p}).
\end{equation*}
Fix a number $K<p$ which will be determined later. For $k<K$ we insert the above estimation into \eqref{eqnCN'CM'} to obtain
\begin{equation*}
\CN'_{\Omega}=\sum_{\substack{k\leq K \\ (k,q)=1}}\mu(k){\rm vol}(\Omega)\cdot\frac{p^n}{k^r a^r}+O_{V,r}(Kp^{n-\frac{1}{2}}\log^r{p})+O\left(\sum_{K<k\leq p}\CM'_{\Omega}(k)\right).
\end{equation*}
The main term is
\begin{align*}
\sum_{\substack{k\leq K \\ (k,a)=1}}\mu(k){\rm vol}(\Omega)\cdot\frac{p^n}{k^r a^r} &= {\rm vol}(\Omega)\cdot\frac{p^n}{a^r}\sum_{\substack{k\leq K \\ (k,a)=1}}\frac{\mu(k)}{k^r} \\
&= {\rm vol}(\Omega)\cdot\frac{p^n}{a^r}\left(\sum_{\substack{k=1 \\ (k,a)=1}}^{\infty}\frac{\mu(k)}{k^r}+O\left(\frac{1}{K^{r-1}}\right)\right) \\
&= {\rm vol}(\Omega)\cdot p^n\left(\frac{1}{\zeta(r)\phi_r(a)}+O\left(\frac{1}{K^{r-1}}\right)\right),
\end{align*}
where the last step can be obtained using the Euler product of the series $\sum_{n=1}^{\infty}\frac{\mu(k)}{k^r}$. The treatment of the error terms is the same as that in Section \ref{secvis}. This completes the proof of the first formula in Theorem \ref{thmvlp}. The second formula of the theorem follows from the above calculation of main term and the estimation of error terms in the second part of Section \ref{secvis}.

\subsection*{Acknowledgements}

The authors are grateful to the referee for many valuable suggestions.


\end{document}